\documentclass{article}

\usepackage{amsmath,amsthm,amssymb,mathtools}
\usepackage[english]{babel}
\usepackage{geometry}
\allowdisplaybreaks[2]
\usepackage[hidelinks]{hyperref}

\numberwithin{equation}{section}

\newtheorem{theorem}{Theorem}[section]
\newtheorem{proposition}[theorem]{Proposition}
\newtheorem{lemma}[theorem]{Lemma}
\newtheorem{corollary}[theorem]{Corollary}

\newcommand{\C}{\mathbb C}
\newcommand{\R}{\mathbb R}
\DeclareMathOperator{\tr}{tr}

\DeclareMathOperator{\HS}{HS}

\title{Pogorelov interior estimates and a Liouville theorem for the
	complex Monge--Amp\`ere equation}

\author{Hongyu Chen \and Jingchen Hu \and Li Sheng}

\date{}

\begin{document}
	\maketitle
	
	\begin{abstract}
		We prove a Pogorelov interior estimate for strictly plurisubharmonic
		solutions of the complex Monge--Amp\`ere equation $\det(u_{i\overline j})=1$ with homogeneous
		Dirichlet data under the condition that, for some constant $\kappa\geq 1$,
		$(\kappa u_{i\overline j}-u_{is}u^{s\overline t} u_{\overline{tj}})$ is non-negative definite. For $\kappa=1$, this gives the estimate
		for real convex solutions.
		The estimate depends only on the dimension, $\kappa$, and the $C^0$
		and $C^1$ norms of the solution. 
		
		 As an application, every smooth entire
		convex solution of $\det(u_{i\bar j})=1$ on $\C^n$ is a real quadratic
		polynomial. The same conclusion holds if real convexity is assumed
		only outside a compact subset of $\C^n$.
	\end{abstract}

	\section{Introduction}\label{sec:introduction}
	
	The classical J\"orgens--Calabi--Pogorelov theorem asserts that a smooth
	convex entire solution of the real Monge--Amp\`ere equation
	\[
	\det D^2v=1\qquad\text{on }\R^m
	\]
	is a quadratic polynomial \cite{r19,r4,r22}. A natural complex analogue
	is the equation
	\begin{equation}\label{e1_1}
		\det(u_{i\bar j})=1\qquad\text{on }\C^n,
	\end{equation} for plurisubharmonic $u$.
	There is, however, an essential difference between the real and complex
	settings. The complex Hessian is unchanged when a pluriharmonic function
	is added to the solution. Thus plurisubharmonicity alone cannot imply
	quadraticity of $u$, and additional conditions are	necessary. Our Liouville theorem assumes convexity with respect to the
	real variables, namely $D_{\R}^2u\ge0$, where $D_{\R}^2u$ denotes the real
	Hessian. For the interior estimate, we allow a weaker matrix bound
	controlling the pure holomorphic Hessian in terms of the complex
	Hessian.
	
	Liouville theorems for \eqref{e1_1} have been studied under both geometric and analytic conditions. Equation \eqref{e1_1} implies that the associated K\"ahler metric $g=(u_{i\bar j})$ is
	Ricci-flat. Calabi asked whether completeness of this metric forces
	it to be flat; see \cite{r5} for this formulation. However, completeness alone is insufficient:
	LeBrun constructed complete nonflat Ricci-flat K\"ahler metrics on
	$\C^n$, $n\ge2$ \cite{LeBrun1991}. Later Tian proved that a complete Calabi--Yau metric on $\C^2$ with maximal 	volume growth is flat, and proposed the corresponding rigidity in
	higher dimensions \cite{Tian2006}. However
	the higher-dimensional analogue fails, as shown independently by
	the constructions of Yang Li, Conlon--Rochon, and Sz\'ekelyhidi
	\cite{YangLi2019,ConlonRochon2021,Szekelyhidi2019}.

	An analytic formulation instead imposes growth conditions on the
	solution. G\'abor Sz\'ekelyhidi asked whether an entire plurisubharmonic
	solution of \eqref{e1_1} satisfying
	\[
	C^{-1}(1+|z|^2)\le u(z)\le C(1+|z|^2)
	\]
	for some $C>0$ must be a quadratic polynomial. This question was
	recorded at the AIM workshop on nonlinear PDEs in real and complex
	geometry; we follow the attribution in \cite{r20} and the problem
	list \cite{AIM2018}. Yu Wang \cite{r29} proved quadraticity under the
	stronger asymptotic condition $u=Q+o(|z|^2)$, where $Q$ is a positive
	Euclidean quadratic model normalized to solve \eqref{e1_1}. His proof
	combines rescaling with Savin's interior
	regularity theorem for small perturbations of solutions of fully
	nonlinear elliptic equations \cite{Savin2007}. Li and Sheng
	\cite{r20} proved that the complex Hessian is constant under completeness
	and a lower quadratic growth bound, and obtained quadraticity under
	completeness and a two-sided quadratic bound. More recently,
	Bing-Long Chen \cite{r5} proved quadraticity under completeness and
	the weaker lower bound
	\[
	u(z)\ge C^{-1}(1+|x|^p),\qquad
	p>\frac{2n}{n+1},\qquad z=x+\sqrt{-1}y,
	\]
	which controls only $n$ totally real directions. 

	Interior estimates also provide rigidity under analytic hypotheses
	on the complex Hessian. Riebesehl and Schulz \cite{r24} established
	interior estimates for third derivatives of mixed type and proved
	that the complex Hessian of an entire solution of \eqref{e1_1} is
	constant when it is uniformly equivalent to the Euclidean metric.
	For equations with a positive, sufficiently regular density,
	B\l ocki and Dinew \cite{r3} obtained local regularity and a priori
	estimates from a $W^{2,p}$ bound with $p>n(n-1)$, while He
	\cite{r13} obtained an interior complex Laplacian bound from an
	$L^p$ bound for the Laplacian with $p>n^2$.
	Cheng and Xu \cite{r9} proved interior $W^{2,p}$ estimates for
	solutions close in $L^\infty$ to a smooth strictly plurisubharmonic
	background when the Monge--Amp\`ere density is sufficiently close
	to $1$. These local results require their respective quantitative
	hypotheses; applying them to an entire solution requires control of
	those hypotheses under rescaling.

	For the real Monge--Amp\`ere equation, Pogorelov's interior Hessian
	estimate with homogeneous Dirichlet data is a basic ingredient in
	the quadratic rigidity theory \cite{r22,r23}. B\l ocki extended
	the interior regularity theory to degenerate equations \cite{r2},
	and Yuan gave a monotonicity proof of Pogorelov's estimate in the
	nondegenerate case \cite{r31}. Interior estimates and Liouville
	theorems have also been developed for other Hessian equations.
	For the real equation $\sigma_2(D^2v)=1$ on the positive branch,
	Warren and Yuan \cite{r30} proved interior Hessian estimates in
	dimension three without a convexity assumption and deduced
	quadraticity of entire solutions with quadratic growth,
	in the sense that $|v(x)|\le C(1+|x|^2)$.
	Shankar and Yuan \cite{r28} obtained the corresponding interior
	estimates and quadratic-growth rigidity in dimension four; in
	higher dimensions their results impose an additional dynamic
	semiconvexity condition.

	In arbitrary dimensions, McGonagle, Song, and Yuan \cite{r21}
	proved interior Hessian estimates for convex solutions of the
	quadratic Hessian equation, and more generally for an almost-convex
	class. The corresponding rigidity for entire almost-convex
	solutions, without a growth assumption, had been proved by Chang
	and Yuan \cite{ChangYuan2010}. Shankar and Yuan subsequently
	established interior Hessian estimates for semiconvex solutions
	\cite{r27} and proved that every smooth semiconvex entire solution
	is quadratic, again without a growth assumption
	\cite{ShankarYuan2022}. Here semiconvexity means that $D^2v\ge-KI$
	for a fixed constant $K\ge0$.
	In a recent preprint, Li and Wu \cite{r33} obtained interior Hessian
	estimates in every dimension without convexity assumptions and
	the resulting Liouville theorem under quadratic growth.
	For positive $C^\alpha$ right-hand sides, the preprints of Zhou and
	Zhu \cite{r34} and Chen, Zhou, and Zhu \cite{r35} establish interior
	$C^{2,\alpha}$ regularity for convex admissible viscosity solutions
	and for solutions on the full positive branch, respectively.
	These variable-density regularity results are distinct from the
	constant-density rigidity statements. Related interior curvature
	estimates for the graphical scalar curvature equation are studied
	in the recent preprint of Qiu and Yan \cite{r36}.

	For higher Hessian equations, Chou and Wang \cite{ChouWang2001}
	developed Pogorelov-type interior estimates in their variational
	theory of the $k$-Hessian equation. For $2\le k<n$, Li, Ren, and Wang
	\cite{LiRenWang2016} proved an estimate of the form
	$(-v)\Delta v\le C$ for $(k+1)$-convex solutions with zero boundary
	values. Applying this estimate to expanding sections, they proved
	that an entire $(k+1)$-convex solution of $\sigma_k(D^2v)=1$ with
	a lower quadratic growth bound is a quadratic polynomial.
	Here $(k+1)$-convexity means that the eigenvalues of $D^2v$ belong
	to the G\aa rding cone $\Gamma_{k+1}$. These results illustrate the
	role of both the structural assumptions in an interior estimate
	and the control of sections needed to obtain global rigidity.

	Our approach to the complex equation is based on constructing
	auxiliary quantities that combine the mixed complex Hessian
	$A=(u_{i\bar j})$ and the pure holomorphic Hessian $B=(u_{ij})$.
	This idea has been used by the authors of the current paper on several topics related to complex Monge-Amp\`ere equations. For degenerate complex Monge--Amp\`ere
	equations, analogous auxiliary functions can be used to obtain metric lower
	bounds, preservation of convexity, and maximum-rank results
	\cite{r15,r14,r16}. The same approach was subsequently
	extended to nondegenerate equations: \cite{r17}
	proved strict real convexity of the Cheng--Yau solution on bounded
	smooth strictly convex domains, and  \cite{r6}
	proved strict convexity of $-\sqrt{-u}$ for the zero-boundary
	solution of $\det(u_{i\bar j})=1$ on such domains in $\C^2$.
	For the related complex $\sigma_2$-Hessian eigenvalue problem,
	Chen, Li, and Ma \cite{ChenLiMa2026} proved that $-\log(-v)$ is
	strictly convex as a real function for the negative first eigenfunction
	$v$ on smooth bounded real uniformly strictly convex domains in
	$\C^n$, $n\ge2$. They also established a Brunn--Minkowski inequality
	for the corresponding first eigenvalue.

	In the present paper, we use the Schur complement
	\[
	M=A-B\overline{A^{-1}}\overline B
	\]
	and the scalar auxiliary quantity
	\[
	S_c=c\,\tr A+\tr\!\left(B\overline{A^{-1}}\overline B\right).
	\]
	The Schur-complement criterion identifies real convexity with
	$M\ge0$; its strict form appears in \cite[Lemma A.6]{r14}.
	We allow the more general condition
	$B\overline{A^{-1}}\overline B\le\kappa A$, with a fixed
	$\kappa\ge1$. The matrix differential identities for $A$ and $B$
	lead to a cancellation of fourth-order derivatives in $LM$,
	where $L=u^{\bar j i}\partial_i\partial_{\bar j}$ is the
	linearized operator. A suitable choice of $c$ then gives a
	positive gradient term in the differential inequality for
	$\log S_c$. This is the key to applying the maximum principle to
	\[
	\Psi=(-u)^\beta S_c\exp\!\left(\alpha|\partial u|^2\right),
	\]
	with $\alpha>0$ chosen in terms of the gradient bound.
	The resulting estimate controls the full real Hessian in terms
	of $n$, $\kappa$, and the $C^0$ and $C^1$ norms of the solution.
	Its constants are independent of the geometry and volume of the
	domain, and no a priori uniform ellipticity or higher-integrability
	bound for the complex Hessian is assumed. We first state this
	general estimate, then its convex specialization, and finally
	the Liouville theorem.
	
	For the statements below,
	$|\partial u|^2:=\sum_{i=1}^n|\partial u/\partial z_i|^2$.
	For a real matrix $H=(H_{ab})$, set $|H|:=\max_{a,b}|H_{ab}|$.
	For a real matrix field $H$ on a set $E$, we use
	\[
	\|H\|_{L^\infty(E)}:=\max_{a,b}\|H_{ab}\|_{L^\infty(E)}.
	\]
	In particular, $\|D_{\R}^2u\|_{L^\infty(E)}$ denotes the
	$L^\infty$ seminorm of the second derivatives.
	
	\begin{theorem}\label{thm:general-interior}
		Let $n\ge2$, let $\kappa\ge1$, and let $\Omega\subset\C^n$ be a domain.
		Suppose that $u\in C^4(\Omega)\cap C^2(\overline\Omega,\R)$ satisfies
		\[\det(u_{i\overline{j}})=1, \text{  in }\Omega,\qquad u=0\quad\text{on }\partial\Omega\]
		\[
		\begin{gathered}
			B\overline{A^{-1}}\overline B\le\kappa A
			\quad\text{in }\Omega. \\
		\text{Here\ }\qquad 	A=(u_{i\bar j}),\qquad B=(u_{ij}).
		\end{gathered}
		\]
		Assume either that $\Omega$ is bounded, or that $\Omega$ and $u$ are
		invariant under translations along a fixed real linear subspace and
		the quotient of $\overline\Omega$ by this subspace is compact.
		Set
		$\beta=(4n-2)\kappa+2$.
		Then
		\begin{equation}\label{e1_general}
				\|(-u)^\beta D_{\R}^2u\|_{L^\infty(\Omega)}
				\le32n\beta(1+\sqrt\kappa)e^{1/8}
				\bigl(1+\|\partial u\|_{L^\infty(\Omega)}^2\bigr)
				\|u\|_{L^\infty(\Omega)}^{\beta-1}.
		\end{equation}
	\end{theorem}
	
	By the Schur-complement criterion, the case $\kappa=1$ is precisely
	real convexity. We therefore obtain the following special case.
	
	\begin{theorem}\label{thm:interior}
		Let $n\ge2$, let $\Omega\subset\C^n$ be a convex domain, and let
		$u\in C^4(\Omega)\cap C^2(\overline\Omega,\R)$ satisfy
		\[
		\det(u_{i\bar j})=1,\qquad D_{\R}^2u\ge0
		\quad\text{in }\Omega,\qquad
		u=0\quad\text{on }\partial\Omega.
		\]
		Assume that either $\Omega$ is bounded, or there is a fixed real linear
		subspace of $\R^{2n}$ contained in $\ker D_{\R}^2u(x)$ for every
		$x\in\Omega$, such that $u$ and $\Omega$ are invariant under translations
		along this subspace and the quotient of $\overline\Omega$ by this
		subspace is compact. Then
		\begin{equation}\label{e1_2}
			\|(-u)^{4n}D_{\R}^2u\|_{L^\infty(\Omega)}
			\le256n^2e^{1/8}
			\bigl(1+\|\partial u\|_{L^\infty(\Omega)}^2\bigr)
			\|u\|_{L^\infty(\Omega)}^{4n-1}.
		\end{equation}
	\end{theorem}
	
	In particular, Theorems~\ref{thm:general-interior}
	and~\ref{thm:interior} give a uniform bound for the full real Hessian
	on every fixed negative sublevel set.
		Applying Theorem~\ref{thm:interior} after a complex-affine normalization
	of convex sections yields the following Liouville theorem.
	Real convexity suffices, without any completeness, growth, or uniform
	ellipticity assumption.
	
	\begin{theorem}\label{thm:liouville}
		Let $n\ge1$ and let $u\in C^\infty(\C^n,\R)$ satisfy
		\[
		D_{\R}^2u\ge0,\qquad
		\det(u_{i\bar j})=1\quad\text{on }\C^n.
		\]
		Then $u$ is a real quadratic polynomial.
	\end{theorem}
	
	Theorem~\ref{thm:liouville} contains the real
	J\"orgens--Calabi--Pogorelov theorem as a special case.
	Indeed, if $v\in C^\infty(\R^n)$ is convex and $\det D^2v=1$,
	then $u(z)=4v(\operatorname{Re}z)$ satisfies
	\[
	u_{i\bar j}=v_{ij},\qquad
	D_{\R}^2u=
	\begin{pmatrix}
		4D^2v&0\\
		0&0
	\end{pmatrix}\ge0.
	\]
	Thus the conclusion of Theorem~\ref{thm:liouville} recovers the
	classical quadratic rigidity of $v$.
	
	By the convexity maximum principle in
	Proposition~\ref{prop:convexity}, real convexity in
	Theorem~\ref{thm:liouville} may be assumed only outside a compact set.
	
	\begin{corollary}\label{cor:outside}
		Let $u\in C^\infty(\C^n,\R)$ satisfy
		$\det(u_{i\bar j})=1$ on $\C^n$.
		Assume that there is a compact set $K_0\Subset\C^n$ such that
		$D_{\R}^2u\ge0$ on $\C^n\setminus K_0$.
		Then $u$ is a real quadratic polynomial.
	\end{corollary}
	
	The paper is organized as follows.
	Section~\ref{sec:identities} derives the differential identities and
	proves Theorems~\ref{thm:general-interior} and~\ref{thm:interior}.
	Section~\ref{sec:sections} establishes the section estimates and the
	complex-affine normalization needed to apply the second-order bound
	on expanding sections. In Section~\ref{sec:liouville} we prove
	Theorem~\ref{thm:liouville} and Corollary~\ref{cor:outside}.
	
	\section{Differential identities and estimates}
	\label{sec:identities}
	
	When $n=1$, the equation reduces to the Laplace equation
	$\Delta_{\mathbb R}u=4$, and Theorem~\ref{thm:liouville} follows from
	Yau's gradient estimate~\cite{r8} applied to the nonnegative harmonic
	functions $\partial_e^2u$, where $e$ is a constant real vector.
	Henceforth we assume $n\ge2$.
	
	Let $\Omega\subset\mathbb C^n$ be open and let
	$u\in C^4(\Omega,\mathbb R)$ be strictly plurisubharmonic with
	$\det(u_{i\bar j})=1$ on $\Omega$.
	Write $z^i=x^i+\sqrt{-1}y^i$ and set
	\[
	\partial_i
	=\frac12(\partial_{x^i}-\sqrt{-1}\partial_{y^i}),
	\qquad
	\partial_{\bar i}
	=\frac12(\partial_{x^i}+\sqrt{-1}\partial_{y^i}).
	\]
	Put
	\[
	A=(u_{i\bar j}),\qquad B=(u_{ij})=B^T.
	\]
	Thus $A$ is Hermitian positive definite, $B$ is complex symmetric,
	and $B^*=\overline B$.
	All indices range from $1$ to $n$, and repeated raised and lowered
	indices are summed. For matrices, $T$, $*$, and an overline denote
	transpose, Hermitian transpose, and entrywise complex conjugation,
	respectively.
	
	We write
	\[
	A^{-1}=(u^{\bar j i}),\qquad
	u_{k\bar j}u^{\bar j i}=\delta_k^i,
	\]
	where $\delta_k^i$ is the Kronecker symbol, and use
	$\overline A^{-1}=\overline{A^{-1}}$.
	Since $A^{-1}$ is Hermitian, the $(i,j)$ entry of
	$\overline{A^{-1}}$ is
	$u^{i\bar j}:=u^{\bar j i}=\overline{u^{\bar i j}}$.
	If $H>0$ is Hermitian, $H^{1/2}$ denotes its positive Hermitian
	square root and $H^{-1/2}$ its inverse.
	
	Define
	\[
	L=u^{\bar q p}\partial_p\partial_{\bar q}.
	\]
	Thus $Lf=u^{\bar q p}f_{p\bar q}$ for a scalar function $f$, and
	$L$ acts componentwise on matrix-valued functions.
	For a covector $\eta=(\eta_i)$, set
	\[
	|\eta|_A^2=u^{\bar j i}\eta_i\overline{\eta_j}.
	\]
	Differentiating $\log\det A=0$ gives the basic first-order identity
	\begin{equation}\label{e2_1}
		u^{\bar q p}u_{p\bar q i}=0.
	\end{equation}
	We next record the two second-order identities that will be used below.
	
	\begin{lemma}\label{lem:matrix-identities}
		The equation $\det A=1$ implies
		\begin{align}
			LA
			&=u^{\bar q p}(\partial_pA)A^{-1}(\partial_{\bar q}A)
			=u^{\bar q p}(\partial_{\bar q}B)
			\overline A^{-1}(\partial_p\overline B),
			\label{e2_2}\\
			LB
			&=u^{\bar q p}(\partial_pA)A^{-1}(\partial_{\bar q}B)
			=u^{\bar q p}(\partial_{\bar q}B)
			\overline A^{-1}(\partial_p\overline A).
			\label{e2_3}
		\end{align}
		Moreover,
		\begin{equation}\label{e2_4}
			u^{\bar q p}\partial_{\bar q}
			\bigl(\overline A^{-1}\partial_p\overline B\bigr)=0.
		\end{equation}
	\end{lemma}
	
	\begin{proof}
		For any coordinate direction $\alpha$, differentiating $AA^{-1}=I$
		gives
		\begin{equation}\label{e2_5}
			\partial_\alpha u^{\bar j i}
			=-u^{\bar j a}(\partial_\alpha A)_{a\bar b}u^{\bar b i}.
		\end{equation}
		Apply $\partial_{\bar j}$ and $\partial_j$, respectively, to
		\eqref{e2_1}, and use \eqref{e2_5}.
		After commuting derivatives, this yields
		\[
		\begin{aligned}
			(LA)_{i\bar j}
			&=u^{\bar q p}u_{i\bar s p}u^{\bar s t}u_{t\bar j\bar q},\\
			(LB)_{ij}
			&=u^{\bar q p}u_{i\bar s p}u^{\bar s t}u_{tj\bar q}.
		\end{aligned}
		\]
		These are the first matrix expressions in \eqref{e2_2} and
		\eqref{e2_3}.
		The $(i,\bar j)$ component of the second expression in \eqref{e2_2}
		is
		\[
		u^{\bar q p}u_{is\bar q}u^{s\bar t}u_{p\bar t\bar j},
		\]
		and the $(i,j)$ component of the second expression in \eqref{e2_3}
		is
		\[
		u^{\bar q p}u_{is\bar q}u^{s\bar t}u_{j\bar t p}.
		\]
		Using the Hermitian symmetry of $A^{-1}$, commuting the purely
		holomorphic and purely anti-holomorphic derivatives, and renaming
		the summed indices $p,q,s,t$ gives the displayed first expressions.
		
		Finally, taking the conjugate of the first expression in
		\eqref{e2_3} and relabeling $p,q$ gives
		\[
		L\overline B
		=u^{\bar q p}(\partial_{\bar q}\overline A)
		\overline A^{-1}(\partial_p\overline B).
		\]
		Since
		\[
		\partial_{\bar q}(\overline A^{-1})
		=-\overline A^{-1}
		(\partial_{\bar q}\overline A)\overline A^{-1},
		\]
		expanding the left-hand side of \eqref{e2_4} now gives zero.
	\end{proof}
	
	The K\"ahler derivative symmetry and \eqref{e2_1} also give
	\[
	\partial_pu^{\bar q p}
	=-u^{\bar q i}u^{\bar j p}\partial_pA_{i\bar j}
	=-u^{\bar q i}\partial_i\log\det A=0.
	\]
	
	\subsection{The Schur complement and weak real convexity}
	
	Define the Schur complement
	\begin{equation}\label{e2_6}
		M=A-B\overline A^{-1}\overline B.
	\end{equation}
	The following block-matrix criterion relates the real Hessian
	to $A$ and $B$.
	
	\begin{lemma}\label{lem:realcomplex}
		Assume $A>0$ and define $M$ by \eqref{e2_6}.
		Then
		\[
		D_{\mathbb R}^2u\ge0
		\quad\Longleftrightarrow\quad
		\begin{pmatrix}
			A&B\\
			\overline B&\overline A
		\end{pmatrix}\ge0
		\quad\Longleftrightarrow\quad
		M\ge0.
		\]
		The same equivalences hold with every $\ge0$ replaced by $>0$.
	\end{lemma}
	
	\begin{proof}
		Write the real Hessian in blocks as
		\[
		D_{\mathbb R}^2u=
		\begin{pmatrix}
			u_{xx}&u_{xy}\\
			u_{xy}^T&u_{yy}
		\end{pmatrix},
		\]
		where
		\[
		u_{xx}=(u_{x^ix^j}),\qquad
		u_{xy}=(u_{x^iy^j}),\qquad
		u_{yy}=(u_{y^iy^j}).
		\]
		Direct differentiation gives
		\[
		\begin{aligned}
			A&=\frac14(u_{xx}+u_{yy})
			+\frac{\sqrt{-1}}4(u_{xy}-u_{xy}^T),\\
			B&=\frac14(u_{xx}-u_{yy})
			-\frac{\sqrt{-1}}4(u_{xy}+u_{xy}^T).
		\end{aligned}
		\]
		One checks by block multiplication that
		\[
		\begin{pmatrix}
			A&B\\
			\overline B&\overline A
		\end{pmatrix}
		=
		\frac14
		\begin{pmatrix}
			I&-\sqrt{-1}I\\
			I&\sqrt{-1}I
		\end{pmatrix}
		\begin{pmatrix}
			u_{xx}&u_{xy}\\
			u_{xy}^T&u_{yy}
		\end{pmatrix}
		\begin{pmatrix}
			I&I\\
			\sqrt{-1}I&-\sqrt{-1}I
		\end{pmatrix}.
		\]
		The first and last block matrices are Hermitian adjoints and
		invertible, so positive semidefiniteness, as well as positive
		definiteness, is preserved by this congruence.
		Since $\overline A>0$, a direct block elimination gives
		\[
		\begin{pmatrix}
			I&-B\overline A^{-1}\\
			0&I
		\end{pmatrix}
		\begin{pmatrix}
			A&B\\
			\overline B&\overline A
		\end{pmatrix}
		\begin{pmatrix}
			I&0\\
			-\overline A^{-1}\overline B&I
		\end{pmatrix}
		=
		\begin{pmatrix}
			M&0\\
			0&\overline A
		\end{pmatrix}.
		\]
		The two outer matrices are invertible and conjugate transposes
		of each other. Hence
		\[
		\begin{pmatrix}
			A&B\\
			\overline B&\overline A
		\end{pmatrix}\ge0
		\quad\Longleftrightarrow\quad
		M\ge0,
		\]
		and likewise in the positive definite case.
		Because $B^T=B$, one has $B^*=\overline B$, so $M$ is Hermitian.
	\end{proof}
	
	For each derivative direction $p$, define
	\begin{equation}\label{e2_7}
		\mathbb B^{(p)}
		=\partial_pB-B\overline A^{-1}\partial_p\overline A.
	\end{equation}
	The parenthesized superscript denotes the derivative direction,
	not a contravariant tensor index.
	The inverse-matrix derivative gives
	\begin{equation}\label{e2_8}
		\begin{aligned}
			\mathbb B^{(p)}\overline A^{-1}
			&=\partial_p(B\overline A^{-1}),\\
			\overline A^{-1}(\mathbb B^{(q)})^*
			&=\partial_{\bar q}(\overline A^{-1}\overline B).
		\end{aligned}
	\end{equation}
	
	\begin{lemma}\label{lem:LM}
		All fourth-order derivatives cancel in $LM$, and
		\begin{equation}\label{e2_9}
			LM=-u^{\bar q p}\mathbb B^{(p)}
			\overline A^{-1}(\mathbb B^{(q)})^*\le0.
		\end{equation}
		Here the last inequality is an inequality of Hermitian matrices.
	\end{lemma}
	
	\begin{proof}
		The same calculation as in \eqref{e2_4}, now using the conjugate
		of \eqref{e2_2}, gives
		\begin{equation}\label{e2_10}
			u^{\bar q p}\partial_{\bar q}
			\bigl(\overline A^{-1}\partial_p\overline A\bigr)=0.
		\end{equation}
		Consequently, \eqref{e2_3} and \eqref{e2_10} imply
		\begin{equation}\label{e2_11}
				u^{\bar q p}\partial_{\bar q}\mathbb B^{(p)}
				={}LB
				-u^{\bar q p}(\partial_{\bar q}B)
				\overline A^{-1}(\partial_p\overline A)-Bu^{\bar q p}\partial_{\bar q}
				\bigl(\overline A^{-1}\partial_p\overline A\bigr)
				=0.
		\end{equation}
		Differentiating \eqref{e2_6} directly gives
		\begin{equation}\label{e2_12}
			\begin{aligned}
				\partial_pM
				={}&\partial_pA-(\partial_pB)\overline A^{-1}\overline B
				+B\overline A^{-1}(\partial_p\overline A)
				\overline A^{-1}\overline B-B\overline A^{-1}\partial_p\overline B\\
				={}&\partial_pA-\mathbb B^{(p)}\overline A^{-1}\overline B
				-B\overline A^{-1}\partial_p\overline B.
			\end{aligned}
		\end{equation}
		Apply $u^{\bar q p}\partial_{\bar q}$ to this identity.
		The product rule gives
		\[
		\begin{aligned}
			LM={}&LA
			-u^{\bar q p}(\partial_{\bar q}\mathbb B^{(p)})
			\overline A^{-1}\overline B
			-u^{\bar q p}\mathbb B^{(p)}
			\partial_{\bar q}(\overline A^{-1}\overline B)\\
			&-u^{\bar q p}(\partial_{\bar q}B)
			\overline A^{-1}(\partial_p\overline B)
			-Bu^{\bar q p}\partial_{\bar q}
			\bigl(\overline A^{-1}\partial_p\overline B\bigr).
		\end{aligned}
		\]
		The second term vanishes by \eqref{e2_11}, and the last by
		\eqref{e2_4}.
		The first and fourth terms cancel by the second expression
		in \eqref{e2_2}.
		Finally, \eqref{e2_8} gives
		\[
		\partial_{\bar q}(\overline A^{-1}\overline B)
		=\overline A^{-1}(\mathbb B^{(q)})^*,
		\]
		which proves the equality in \eqref{e2_9}.
	\end{proof}
	
	We have the following convexity maximum principle.
	
	\begin{proposition}\label{prop:convexity}
		Let $\Omega\Subset\mathbb C^n$ be a connected domain and let
		$u\in C^\infty(\overline\Omega,\mathbb R)$ satisfy
		$(u_{i\bar j})>0$ and $\det(u_{i\bar j})=1$ on $\overline\Omega$.
		If $D_{\mathbb R}^2u\ge0$ on $\partial\Omega$, then
		$D_{\mathbb R}^2u\ge0$ on $\overline\Omega$.
	\end{proposition}
	
	\begin{proof}
		By Lemma~\ref{lem:realcomplex}, the boundary hypothesis is
		equivalent to $M\ge0$ on $\partial\Omega$.
		Fix a constant vector $\xi\in\mathbb C^n$ and set
		$f_\xi=\xi^*M\xi$.
		Lemma~\ref{lem:LM} gives $Lf_\xi\le0$.
		Since $A>0$ on the compact set $\overline\Omega$, the operator $L$
		is uniformly elliptic there.
		The minimum principle and $f_\xi\ge0$ on $\partial\Omega$ imply
		$f_\xi\ge0$ in $\Omega$.
		Since $\xi$ is arbitrary, $M\ge0$ on $\overline\Omega$, and
		Lemma~\ref{lem:realcomplex} gives $D_{\mathbb R}^2u\ge0$.
	\end{proof}
	
	We record the two global consequences of \eqref{e2_9} needed
	for an entire weakly convex solution.
	
	\begin{lemma}\label{lem:rank}
		Let $u\in C^\infty(\mathbb C^n,\mathbb R)$ satisfy
		$D_{\mathbb R}^2u\ge0$ and $\det A=1$.
		Then both $\operatorname{rank}M$ and
		$\operatorname{rank}D_{\mathbb R}^2u$ are constant on $\mathbb C^n$.
	\end{lemma}
	
	\begin{proof}
		Since $D_{\mathbb R}^2u\ge0$ and $\det A=1$, the matrix $A$ is
		positive definite.
		Lemma~\ref{lem:realcomplex} then gives $M\ge0$.
		Fix $x_0\in\mathbb C^n$ and $\xi\in\ker M(x_0)$, and regard
		$\xi$ as a constant vector.
		The function $f_\xi=\xi^*M\xi$ satisfies
		$f_\xi\ge0$, $f_\xi(x_0)=0$, and $Lf_\xi\le0$ by \eqref{e2_9}.
		
		Since $A$ is smooth and positive definite, $L$ is uniformly
		elliptic on sufficiently small balls.
		At any zero of $f_\xi$, the strong minimum principle shows that
		$f_\xi$ vanishes in a neighborhood of that point.
		Thus the zero set of $f_\xi$ is nonempty, open, and closed.
		The connectedness of $\mathbb C^n$ gives $f_\xi\equiv0$.
		Since $M\ge0$, this implies
		$M(x)\xi=0$ for every $x\in\mathbb C^n$.
		
		Taking $x_0=x$ in this argument gives
		$\ker M(x)\subset\ker M(y)$ for all $x,y\in\mathbb C^n$.
		Applying the same inclusion with $x$ and $y$ reversed gives
		$\ker M(x)=\ker M(y)$.
		Therefore $\operatorname{rank}M$ is constant.
		
		By the same block elimination as above,
		\[
		\begin{pmatrix}
			I&-B\overline A^{-1}\\
			0&I
		\end{pmatrix}
		\begin{pmatrix}
			A&B\\
			\overline B&\overline A
		\end{pmatrix}
		\begin{pmatrix}
			I&0\\
			-\overline A^{-1}\overline B&I
		\end{pmatrix}
		=
		\begin{pmatrix}
			M&0\\
			0&\overline A
		\end{pmatrix}.
		\]
		Since the two outer matrices are invertible and $\overline A$
		has rank $n$,
		\[
		\operatorname{rank}
		\begin{pmatrix}
			A&B\\
			\overline B&\overline A
		\end{pmatrix}
		=n+\operatorname{rank}M.
		\]
		It is congruent to $D_{\mathbb R}^2u$ by the explicit formula
		in the proof of Lemma~\ref{lem:realcomplex}; hence
		$\operatorname{rank}D_{\mathbb R}^2u$ is constant as well.
	\end{proof}
	
For the solution of Theorem~\ref{thm:liouville},
Lemma~\ref{lem:rank} gives
$\operatorname{rank}D_{\R}^2u\equiv r$.
We use the following consequence of Caffarelli--Nirenberg--Spruck
\cite[(4) and Lemma~2]{r25}: if $f$ is an entire convex solution of
$\det D^2f=0$ on $\mathbb R^N$, then there exists a nonzero vector
$e\in\mathbb R^N$ such that
$$
D^2f(x)e=0
\qquad\text{for every }x\in\mathbb R^N.
$$
If $r<2n$, applying this result to $u$ gives a nonzero fixed direction
$e$ with $D_{\R}^2u(x)e=0$ for every $x\in\C^n$.
Hence $\partial_eu$ is constant, and $u$ splits off an affine variable
in the direction $e$.
The remaining entire convex function still has Hessian rank $r$.
Iterating this argument splits off $2n-r$ affine directions, while in
the remaining $r$ variables the Hessian is positive definite.
Thus, including the case $r=2n$, there is a fixed real subspace
\begin{equation}\label{e2_16}
	\mathcal K:=\ker D_{\R}^2u(x)\qquad(x\in\C^n).
\end{equation}

	We next record the finite-dimensional inequality used in
	the second-order estimate below.
	For a complex matrix $T$, write
	$\|T\|_{\HS}^2=\tr(TT^*)$.
	
	\begin{lemma}\label{lem:tracefree}
		Let $c\ge0$, let $C\in\C^{n\times n}$ satisfy $\tr C=0$,
		and let $e,k\in\C^n$ satisfy $|e|=1$ and $|k|\le1$.
		Then
		\begin{equation}\label{e2_18}
			|cCe+C^*k|^2
			\le(c^2+n-1)\|C\|_{\mathrm{HS}}^2.
		\end{equation}
	\end{lemma}
	
	\begin{proof}
		Put $w=cCe+C^*k$.
		For a unit vector $v\in\C^n$ and $\theta\in\R$, set
		\[
		Q_\theta
		=ce^{\sqrt{-1}\theta}ve^*
		+e^{-\sqrt{-1}\theta}kv^*.
		\]
		With the real Hilbert--Schmidt pairing
		$\langle X,Y\rangle_{\R}=\operatorname{Re}\tr(XY^*)$,
		one has
		\[
		\operatorname{Re}(e^{-\sqrt{-1}\theta}v^*w)
		=\operatorname{Re}\tr(CQ_\theta^*).
		\]
		Since $\tr C=0$, the scalar part of $Q_\theta$ does not contribute.
		Thus, with
		\[
		Q_{\theta,0}=Q_\theta-\frac{\tr Q_\theta}{n}I,
		\]
		we obtain
		\[
		\operatorname{Re}(e^{-\sqrt{-1}\theta}v^*w)
		\le\|C\|_{\mathrm{HS}}\|Q_{\theta,0}\|_{\mathrm{HS}}.
		\]
		It remains to estimate the last norm.
		Put $r=|e^*v|$ and $s=|v^*k|$.
		Since $0\le r,s\le1$,
		\[	\|Q_{\theta,0}\|_{\mathrm{HS}}^2
			=\|Q_\theta\|_{\mathrm{HS}}^2
			-\frac{|\tr Q_\theta|^2}{n}\le
			c^2+1+
			\frac{-c^2r^2-s^2+2c(n-1)rs}{n}.
		\]
		The difference between $c^2+n-1$ and the right-hand side is
		\[
		\frac{(cr-(n-1)s)^2+n(n-2)(1-s^2)}{n}\ge0.
		\]
		Since
		\[
		|w|
		=\sup_{|v|=1,\,\theta\in\R}
		\operatorname{Re}(e^{-\sqrt{-1}\theta}v^*w),
		\]
		taking the supremum over $v$ and $\theta$ proves \eqref{e2_18}.
	\end{proof}
	
	We now prove Theorem~\ref{thm:general-interior}.
	In the cylindrical case all functions and coefficients are invariant
	under translations in the fixed subspace, and compactness below is
	understood on the quotient.
	
	\begin{proof}[Proof of Theorem~\ref{thm:general-interior}]
		Set
		\[
		K=\|\partial u\|_{L^\infty(\Omega)},\qquad
		U=\|u\|_{L^\infty(\Omega)}.
		\]
		We use $A$, $B$, $L$, and $\mathbb B^{(p)}$ as defined above.
		Strict plurisubharmonicity and the maximum principle give $u<0$
		in $\Omega$.
		In the cylindrical case, a maximum is attained on the compact quotient
		and the same principle applies at a lift to $\Omega$.
		Set
		\[
		a=\tr A,\qquad d=\tr(B\overline A^{-1}\overline B).
		\]
		The assumed matrix inequality gives $0\le d\le\kappa a$.
		The differentiated identities used below require only $A>0$ and
		$\det A=1$; they do not require real convexity.
		
		Taking the trace in \eqref{e2_2} and \eqref{e2_9}, set
		\[
		\begin{aligned}
			E=u^{\bar q p}
			\tr\bigl((\partial_p A)A^{-1}(\partial_{\bar q}A)\bigr),\qquad
			F=u^{\bar q p}
			\tr\bigl(\mathbb B^{(p)}
			\overline A^{-1}(\mathbb B^{(q)})^*\bigr).
		\end{aligned}
		\]
		Both quantities are nonnegative and
		\begin{equation}\label{e2_19}
			La=E,\qquad Ld=E+F.
		\end{equation}
		For $c\ge0$, set $S_c=ca+d$.
		Then
		\begin{equation}\label{e2_20}
			LS_c=(c+1)E+F.
		\end{equation}
		
		We next estimate $\partial S_c$.
		Differentiating $d=\tr(B\overline A^{-1}\overline B)$ and
		using \eqref{e2_7} gives
		\begin{equation}\label{e2_21}
			d_i
			=\tr(\mathbb B^{(i)}\overline A^{-1}\overline B)
			+\tr(B\overline A^{-1}\partial_i\overline B).
		\end{equation}
		The first term satisfies
		\begin{equation}\label{e2_22}
			u^{\bar j i}
			\tr(\mathbb B^{(i)}\overline A^{-1}\overline B)
			\overline{
				\tr(\mathbb B^{(j)}\overline A^{-1}\overline B)}
			\le dF.
		\end{equation}
		Indeed, after diagonalizing $A$, this is the Hilbert--Schmidt
		Cauchy--Schwarz inequality applied to
		$\mathbb B^{(i)}\overline A^{-1/2}$ and
		$B\overline A^{-1/2}$, followed by summation with the weights
		from $A^{-1}$.
		
		For the remaining term, diagonalize
		$A=\operatorname{diag}(\lambda_1,\ldots,\lambda_n)$ at the
		point under consideration and define, only for this calculation,
		\[
		H_{piq}
		=\frac{u_{pi\bar q}}{\sqrt{\lambda_p\lambda_i\lambda_q}},
		\qquad
		C^{(p)}_{iq}=H_{piq}.
		\]
		Differentiating $\log\det A=0$ in the $p$ direction gives
		$\tr C^{(p)}=0$.
		Moreover, the symmetry in the first two holomorphic indices gives
		\begin{equation}\label{e2_23}
			E=\sum_p\lambda_p\|C^{(p)}\|_{\mathrm{HS}}^2.
		\end{equation}
		By the matrix bound,
		\[
		A^{-1/2}B\overline A^{-1}\overline B A^{-1/2}
		\le\kappa I.
		\]
		Thus, if $e_p$ is the $p$-th coordinate vector and $k^{(p)}$
		is the transpose of the $p$-th row of
		$A^{-1/2}B\overline A^{-1/2}$, then
		$|k^{(p)}|\le\sqrt\kappa$.
		Direct substitution gives
		\begin{equation}\label{e2_24}
			\begin{aligned}
				&A^{-1/2}
				\Bigl(
				c\,\partial a+
				\bigl(\tr(B\overline A^{-1}\partial_i\overline B)\bigr)_i
				\Bigr)
				=\sum_p\lambda_p
				\bigl(cC^{(p)}e_p+(C^{(p)})^*k^{(p)}\bigr).
			\end{aligned}
		\end{equation}
		Componentwise,
		\[
		\begin{aligned}
			\frac{a_i}{\sqrt{\lambda_i}}
			&=\sum_p\lambda_p(C^{(p)}e_p)_i,\quad 
			\frac{\tr(B\overline A^{-1}\partial_i\overline B)}
			{\sqrt{\lambda_i}}
			&=\sum_p\lambda_p
			\bigl((C^{(p)})^*k^{(p)}\bigr)_i.
		\end{aligned}
		\]
		Applying Lemma~\ref{lem:tracefree} with $c/\sqrt\kappa$
		and $k^{(p)}/\sqrt\kappa$, and multiplying by $\kappa$, gives
		\[
		|cC^{(p)}e_p+(C^{(p)})^*k^{(p)}|^2
		\le
		\bigl(c^2+(n-1)\kappa\bigr)
		\|C^{(p)}\|_{\mathrm{HS}}^2.
		\]
		Weighted Cauchy--Schwarz and \eqref{e2_23} therefore give
		\begin{equation}\label{e2_25}
			\left|
			c\,\partial a+
			\bigl(\tr(B\overline A^{-1}\partial_i\overline B)\bigr)_i
			\right|_A^2
			\le
			\bigl(c^2+(n-1)\kappa\bigr)aE.
		\end{equation}
		Combining \eqref{e2_21}, \eqref{e2_22}, and \eqref{e2_25},
		we obtain
		\[
		|\partial S_c|_A
		\le
		\sqrt{\bigl(c^2+(n-1)\kappa\bigr)aE}
		+\sqrt{dF}.
		\]
		Hence
		\begin{equation}\label{e2_26}
			|\partial S_c|_A^2
			\le
			\left(
			\frac{c^2+(n-1)\kappa}{c+1}a+d
			\right)
			\bigl((c+1)E+F\bigr).
		\end{equation}
		Assume $c>(n-1)\kappa$.
		The function
		\[
		t\longmapsto
		\frac{(c^2+(n-1)\kappa)/(c+1)+t}{c+t}
		\]
		is increasing on $[0,\kappa]$, since its derivative is
		$\frac{c-(n-1)\kappa}{(c+1)(c+t)^2}>0$.
		Since $0\le d/a\le\kappa$, \eqref{e2_20} therefore gives
		\[
		|\partial S_c|_A^2
		\le
		\frac{c^2+c\kappa+n\kappa}{(c+1)(c+\kappa)}
		S_cLS_c.
		\]
		It follows that
		\begin{equation}\label{e2_27}
			L\log S_c
			\ge
			\frac{c-(n-1)\kappa}{c^2+c\kappa+n\kappa}
			|\partial\log S_c|_A^2.
		\end{equation}
		Choose
		\[
		c=(2n-1)\kappa,\qquad
		\beta=2c+2=(4n-2)\kappa+2,\qquad
		S=ca+d.
		\]
		Then \eqref{e2_27} becomes
		\begin{equation}\label{e2_28}
			L\log S
			\ge\frac1{\beta-1}|\partial\log S|_A^2.
		\end{equation}
		
		Set $G=|\partial u|^2$.
		Since $Lu=n$,
		\begin{equation}\label{e2_29}
			L\log(-u)
			=-\frac n{-u}
			-\frac{|\partial u|_A^2}{(-u)^2}.
		\end{equation}
		A direct differentiation, using \eqref{e2_1}, gives
		\begin{equation}\label{e2_30}
			LG=a+d,\qquad
			\partial G=A\,\partial u+B\,\overline{\partial u}.
		\end{equation}
		Since $|\partial u|\le K$,
		\begin{equation}\label{e2_31}
			|\partial G|_A^2\le4K^2(a+d).
		\end{equation}
		Indeed,
		\[
		|A\partial u|_A^2\le K^2a,\qquad
		|B\overline{\partial u}|_A^2\le K^2d.
		\]
		Hence $|x+y|^2\le2|x|^2+2|y|^2$ gives
		$|\partial G|_A^2\le2K^2(a+d)$, which implies \eqref{e2_31}.
		
		Let $\alpha=\frac1{8(1+K^2)}$ and consider
		\[
		\Psi=(-u)^\beta Se^{\alpha G}.
		\]
		It vanishes on the boundary and is invariant along the fixed
		subspace in the cylindrical case.
		By continuity, $A\ge0$ on the boundary, and $\det A=1$ implies
		$A>0$ there.
		Since $u\in C^2(\overline\Omega)$, compactness of
		$\overline\Omega$, or of its quotient, bounds
		$A$, $A^{-1}$, $B$, $S$, and $G$.
		Thus $\Psi$ attains a positive interior maximum $x_0$, either
		in $\Omega$ or on the compact quotient.
		A lift of a quotient maximum is a local maximum in $\Omega$,
		since $\Psi$ is translation invariant.
		At $x_0$,
		\[
		\partial\log S
		=\beta\frac{\partial u}{-u}-\alpha\partial G.
		\]
		Using \eqref{e2_28}, \eqref{e2_29}, and \eqref{e2_30},
		we obtain
		\begin{align*}
			0&\ge L\log\Psi\\
			&\ge
			-\frac{n\beta}{-u}
			-\beta\left|\frac{\partial u}{-u}\right|_A^2
			+\frac1{\beta-1}
			\left|
			\beta\frac{\partial u}{-u}-\alpha\partial G
			\right|_A^2
			+\alpha(a+d).
		\end{align*}
		For arbitrary covectors $X,Y$,
		\[
		\frac1{\beta-1}|\beta X-Y|_A^2-\beta|X|_A^2
		=
		\frac{\beta}{\beta-1}|X-Y|_A^2-|Y|_A^2
		\ge-|Y|_A^2.
		\]
		Therefore \eqref{e2_31} gives
		\[
		0\ge
		-\frac{n\beta}{-u}
		+\alpha(1-4\alpha K^2)(a+d).
		\]
		Since $4\alpha K^2\le1/2$,
		\[
		(-u)(a+d)
		\le16n\beta(1+K^2)
		\qquad\text{at }x_0.
		\]
		As $c\ge1$, we have $ca\le S\le c(a+d)$.
		Using $-u\le U$ and $\alpha G\le1/8$, we obtain
		\[
		\Psi(x_0)
		\le16n\beta c\,e^{1/8}(1+K^2)U^{\beta-1}.
		\]
		By maximality, for every $x\in\Omega$,
		\begin{equation}\label{e2_32}
			S(x)
			\le
			16n\beta c\,e^{1/8}(1+K^2)
			U^{\beta-1}(-u(x))^{-\beta}.
		\end{equation}
		
		Finally, positivity of $A$ gives
		$|A_{ij}|\le\sqrt{A_{ii}A_{jj}}$.
		By the Cauchy--Schwarz inequality and the assumed matrix inequality,
		\[
		|B_{ij}|^2
		\le\bigl(B\overline A^{-1}\overline B\bigr)_{ii}A_{jj}
		\le\kappa A_{ii}A_{jj}.
		\]
		The relation between the real and complex Hessians therefore yields
		\[
		\begin{aligned}
			|D_{\R}^2u|
			&\le2\max_{i,j}\bigl(|A_{ij}|+|B_{ij}|\bigr)\\
			&\le2(1+\sqrt\kappa)a
			\le\frac{2(1+\sqrt\kappa)}cS.
		\end{aligned}
		\]
		Combining this with \eqref{e2_32} proves \eqref{e1_general}.
	\end{proof}
	
	\begin{proof}[Proof of Theorem~\ref{thm:interior}]
		Real convexity gives $A\ge0$, and $\det A=1$ then gives $A>0$.
		By Lemma~\ref{lem:realcomplex},
		\[
		B\overline A^{-1}\overline B\le A.
		\]
		Apply Theorem~\ref{thm:general-interior} with $\kappa=1$.
		Then $\beta=4n$ and
		$32n\beta(1+\sqrt\kappa)=256n^2$, proving \eqref{e1_2}.
	\end{proof}
	
	\section{Estimates on convex sections and rescaling}
	\label{sec:sections}
	
	We now apply Theorem~\ref{thm:interior} to normalized convex sections.
	The normalization provides the gradient bound required in
	\eqref{e1_2} without controlling the Euclidean diameter of the section.
	
	\begin{lemma}\label{lem:section-gradient}
		Let $\Omega\subset\mathbb C^n$ be a smooth convex domain and let
		$v\in C^2(\overline\Omega,\mathbb R)$ satisfy
		\[
		\begin{gathered}
			-1\le v<0\quad\text{in }\Omega,\qquad
			v=0\quad\text{on }\partial\Omega,\\
			D_{\mathbb R}^2v\ge0,\qquad
			\det(v_{i\bar j})=1.
		\end{gathered}
		\]
		Assume that there is a real convex quadratic function $q$ such that
		\[
		0\le q\le4n^2\quad\text{on }\overline\Omega,\qquad
		I\le(q_{i\bar j})\le4n^2I.
		\]
		Assume that either $\Omega$ is bounded, or $\Omega$, $v$, and $q$
		are invariant under translations along the same fixed real subspace
		and the quotient of $\overline\Omega$ by this subspace is compact.
		Then there is a constant $c_n>0$ such that
		\begin{equation}\label{e3_1}
			\operatorname{dist}_{\mathbb R}
			\bigl(\{v\le-1/4\},\partial\Omega\bigr)\ge c_n
		\end{equation}
		and
		\begin{equation}\label{e3_2}
			\sup_{\{v\le-1/4\}}|\partial v|\le c_n^{-1}.
		\end{equation}
	\end{lemma}
	
	\begin{proof}
		Since $D_{\mathbb R}^2v\ge0$ and $\det(v_{i\bar j})=1$,
		one has $(v_{i\bar j})>0$.
		Set
		\[
		\varepsilon=\frac1{32n^2},\qquad
		\delta^2=\frac{\varepsilon^{n-1}}2.
		\]
		Fix $x\in\{v\le-1/4\}$ and a nearest boundary point $y$.
		Put $d=|x-y|$ and $\nu=(x-y)/d$.
		By convexity, the affine function
		\[
		l(z)=\nu\cdot(z-y)
		\]
		satisfies $l\ge0$ on $\overline\Omega$, $l(x)=d$,
		and $|\partial l|^2=1/4$.
		In the cylindrical case, the supporting normal $\nu$ is orthogonal
		to the invariant subspace, so $l$ has the same translation
		invariance as $v$ and $q$.
		
		On $\Omega\cap\{0<l<\delta\}$ consider
		\[
		b=\varepsilon(q-4n^2)
		-\frac{2l}{\delta}
		+\frac{l^2}{\delta^2}.
		\]
		Since $l$ is affine,
		\[
		(b_{i\bar j})
		=\varepsilon(q_{i\bar j})
		+\frac2{\delta^2}l_i l_{\bar j}
		\ge
		\varepsilon I+\frac2{\delta^2}l_i l_{\bar j}.
		\]
		The rank-one determinant formula and $|\partial l|^2=1/4$ give
		\[
		\det(b_{i\bar j})
		\ge
		\varepsilon^n+\frac{\varepsilon^{n-1}}{2\delta^2}
		=\varepsilon^n+1>1.
		\]
		On $\partial\Omega$ in the strip, $b\le0=v$, while on
		$l=\delta$, $b\le-1\le v$.
		Thus the comparison principle gives $b\le v$.
		In the cylindrical case, this follows by lifting a maximum of
		$b-v$ from the compact quotient: a positive interior maximum would
		give $(b_{i\bar j})\le(v_{i\bar j})$, contradicting their determinants.
		If $d<\delta$, then
		\[
		v(x)\ge b(x)\ge-\frac18-\frac{2d}{\delta}.
		\]
		Since $v(x)\le-1/4$, it follows that $d\ge\delta/16$.
		The same bound is immediate if $d\ge\delta$.
		This proves \eqref{e3_1} with $c_n=\delta/16$.
		
		If $v(x)\le-1/4$, then $B_{c_n/2}(x)\subset\Omega$.
		For every real unit vector $e$, convexity gives
		\[
		D_{\mathbb R}v(x)\cdot e
		\le
		\frac{v(x+(c_n/2)e)-v(x)}{c_n/2}
		\le\frac2{c_n}.
		\]
		Applying the same inequality to $-e$ gives
		$|D_{\mathbb R}v(x)|\le2c_n^{-1}$, and therefore
		$|\partial v(x)|\le c_n^{-1}$.
		This proves \eqref{e3_2}.
	\end{proof}
	
	\begin{proposition}\label{prop:globalbound}
		Let $n\ge2$ and let $u\in C^\infty(\mathbb C^n,\mathbb R)$ satisfy
		\[
		D_{\mathbb R}^2u\ge0,\qquad
		\det(u_{i\bar j})=1\quad\text{on }\mathbb C^n.
		\]
		Then $\|D_{\mathbb R}^2u\|_{L^\infty(\mathbb C^n)}<\infty$.
	\end{proposition}
	
	\begin{proof}
		By Lemma~\ref{lem:rank} and the splitting argument in
		Section~\ref{sec:identities}, the real Hessian has the fixed nullspace
		$\mathcal K$ from \eqref{e2_16}.
		After subtracting the tangent affine function at the origin and
		making a constant complex-linear change of coordinates with
		determinant of modulus one, we may assume
		\begin{equation}\label{e3_3}
			u(0)=0,\qquad D_{\mathbb R}u(0)=0,\qquad A(0)=I.
		\end{equation}
		Indeed, $\det A(0)=1$, so the complex-linear change may be chosen
		with determinant of modulus one.
		We continue to denote the transformed fixed nullspace by
		$\mathcal K$.
		These operations preserve all hypotheses.
		Since $u$ is convex, \eqref{e3_3} gives $u\ge0$.
		Moreover,
		\[
		D_{\mathbb R}(\partial_eu)=D_{\mathbb R}^2u\,e=0
		\]
		shows that $\partial_eu$ is constant for $e\in\mathcal K$;
		the normalization $D_{\mathbb R}u(0)=0$ therefore gives
		$\partial_eu\equiv0$.
		Hence $u$ is invariant under translations in $\mathcal K$.
		
		Let $\mathcal V=\mathcal K^\perp$ and, for $h>0$, set
		\[
		S_h=\{u<h\}.
		\]
		Then
		\[
		S_h=(S_h\cap\mathcal V)+\mathcal K.
		\]
		The restriction of $D_{\mathbb R}^2u$ to $\mathcal V$ is positive
		definite, so $u|_{\mathcal V}$ is strictly convex.
		The compact unit sphere in $\mathcal V$ gives
		\[
		\min_{\substack{\xi\in\mathcal V\\|\xi|=1}}u(\xi)>0,
		\]
		and convexity yields linear growth along every ray outside the
		unit ball.
		Thus $S_h\cap\mathcal V$ is bounded.
		Since the only critical point of $u|_{\mathcal V}$ is its minimum
		at the origin, the level $u=h$ is regular for every $h>0$;
		hence $\partial S_h$ is smooth.
		
		Apply John's ellipsoid theorem~\cite{r18} in the real vector space
		$\mathcal V$ to the convex body $\overline{S_h\cap\mathcal V}$.
		Since $\dim_{\mathbb R}\mathcal V\le2n$, there are
		$a_h\in\mathcal V$ and a nonnegative real quadratic function $q_h$,
		with $q_h(a_h)=0$ and real Hessian kernel $\mathcal K$, such that
		\begin{equation}\label{e3_4}
			\{q_h<1\}\subset S_h\subset\{q_h<4n^2\}.
		\end{equation}
		Put
		\[
		C_h=((q_h)_{i\bar j}).
		\]
		The nonnegative symmetric matrices $D_{\mathbb R}^2q_h$ and
		$D_{\mathbb R}^2u(0)$ have the same kernel $\mathcal K$.
		Thus there exists $\varepsilon_h>0$ such that
		\[
		D_{\mathbb R}^2q_h\ge\varepsilon_hD_{\mathbb R}^2u(0).
		\]
		The corresponding inequality for the complex Hessians, together
		with \eqref{e3_3}, gives
		\[
		C_h\ge\varepsilon_h A(0)=\varepsilon_h I>0.
		\]
		
		Set $c_h=(\det C_h)^{1/n}$.
		On $\{q_h<1\}$, both $u-h$ and $(q_h-1)/c_h$ have complex
		Monge--Amp\`ere determinant one, and
		$u-h\le0=(q_h-1)/c_h$ on the boundary.
		The comparison principle gives
		\[
		u-h\le\frac{q_h-1}{c_h}.
		\]
		Similarly, on $S_h$,
		\[
		\frac{q_h-4n^2}{c_h}\le u-h.
		\]
		In the translation-invariant case, the difference of the two
		functions descends to the compact quotient and attains its extrema
		there.
		Any lift of a quotient extremum is a local extremum in the ambient
		domain, since the difference is constant along the invariant
		directions.
		The usual maximum principle may therefore be applied after lifting.
		Indeed, in the ambient domain the difference satisfies the uniformly
		elliptic linear equation obtained by integrating the linearization
		of $\log\det$ along the segment joining the two complex Hessians.
		
		Evaluating the first inequality at $a_h$ and using $u(a_h)\ge0$
		gives $hc_h\ge1$.
		Evaluating the second at the origin gives
		\[
		hc_h\le4n^2-q_h(0)\le4n^2.
		\]
		Hence
		\begin{equation}\label{e3_5}
			1\le hc_h\le4n^2.
		\end{equation}
		Since $C_h>0$, set
		\[
		G_h=c_h^{1/2}\overline{C_h}^{-1/2}.
		\]
		Then
		\begin{equation}\label{e3_6}
			G_h^TC_h\overline{G_h}=c_h I.
		\end{equation}
		Taking determinants in \eqref{e3_6} gives $|\det G_h|=1$.
		Introduce the affine coordinates
		\[
		z=a_h+\sqrt h\,G_hw
		\]
		and set
		\[
		\begin{aligned}
			\Omega_h
			&=\{w:a_h+\sqrt h\,G_hw\in S_h\},\\
			v_h(w)
			&=\frac{u(a_h+\sqrt h\,G_hw)}h-1.
		\end{aligned}
		\]
		The chain rule gives
		\[
		((v_h)_{i\bar j})(w)
		=G_h^TA(a_h+\sqrt h\,G_hw)\overline{G_h},
		\]
		so
		\[
		\det((v_h)_{i\bar j})=1,\qquad
		-1\le v_h<0,\qquad
		v_h=0\quad\text{on }\partial\Omega_h.
		\]
		The domain $\Omega_h$ is smooth and convex.
		If $\mathcal K\ne\{0\}$, then $\Omega_h$, $v_h$, and
		$w\mapsto q_h(a_h+\sqrt h\,G_hw)$ are invariant under translations
		along the transformed fixed nullspace, and the quotient closure
		is compact.
		Moreover, \eqref{e3_4}, \eqref{e3_6}, and \eqref{e3_5} give
		\[
		0\le q_h(a_h+\sqrt h\,G_hw)\le4n^2
		\quad\text{on }\overline{\Omega_h}
		\]
		and
		\[
		\begin{gathered}
			\Bigl(q_h(a_h+\sqrt h\,G_hw)\Bigr)_{i\bar j}
			=hG_h^TC_h\overline{G_h}=hc_hI,\\
			I\le hc_hI\le4n^2I.
		\end{gathered}
		\]
		Here and below, $C_n>0$ denotes a constant depending only on $n$,
		whose value may change from line to line.
		Lemma~\ref{lem:section-gradient} therefore yields
		\begin{equation}\label{e3_7}
			\sup_{\{v_h\le-1/4\}}|\partial v_h|\le C_n.
		\end{equation}
		
		Apply Theorem~\ref{thm:interior} to $v_h+1/4$ on the convex
		sublevel set $\{v_h<-1/4\}$.
		On the quotient, $v_h$ is strictly convex and its only critical value
		is $-1$, so this sublevel set has smooth boundary.
		The $C^0$ norm of $v_h+1/4$ is $3/4$, its gradient is bounded by
		\eqref{e3_7}, and its boundary value is zero.
		If the fixed nullspace is nontrivial, the second alternative in
		Theorem~\ref{thm:interior} applies.
		Since $-(v_h+1/4)\ge1/4$ on $\{v_h\le-1/2\}$, we obtain
		\begin{equation}\label{e3_8}
			\|D_{\mathbb R}^2v_h\|_{L^\infty(\{v_h\le-1/2\})}\le C_n.
		\end{equation}
		
		At the point
		\[
		w=-h^{-1/2}G_h^{-1}a_h,
		\]
		which corresponds to the origin, \eqref{e3_3} gives $v_h(w)=-1$
		and
		\[
		((v_h)_{i\bar j})(w)=G_h^T\overline{G_h}.
		\]
		The real Hessian bound \eqref{e3_8} gives
		\[
		4\operatorname{tr}(G_h^T\overline{G_h})
		=\Delta_{\mathbb R}v_h(w)\le2nC_n.
		\]
		Since
		\[
		G_h^T\overline{G_h}
		=(\overline{G_h})^*\overline{G_h},
		\]
		its eigenvalues are the squares of the singular values of $G_h$.
		Thus all singular values of $G_h$ are bounded above by a dimensional
		constant.
		Since $|\det G_h|=1$, their product is one, and hence
		\begin{equation}\label{e3_9}
			\|G_h\|_{\mathrm{HS}}+\|G_h^{-1}\|_{\mathrm{HS}}\le C_n.
		\end{equation}
		
		Now let $z\in S_{h/2}$ and write
		$z=a_h+\sqrt h\,G_hw$.
		Then $v_h(w)\le-1/2$.
		Since
		\[
		((v_h)_{i\bar j})(w)=G_h^TA(z)\overline{G_h},
		\]
		\eqref{e3_8} gives
		\[
		4\operatorname{tr}
		\bigl(G_h^TA(z)\overline{G_h}\bigr)
		=\Delta_{\mathbb R}v_h(w)\le2nC_n.
		\]
		The matrix $G_h^TA(z)\overline{G_h}$ is positive Hermitian,
		so its eigenvalues are bounded above.
		Together with \eqref{e3_9}, this gives
		\[
		A(z)\le C_nI.
		\]
		As the sections $S_{h/2}$ exhaust $\mathbb C^n$, this gives a
		global upper bound for $A$.
		The identity $\det A=1$ then gives a global positive lower bound
		for $A$.
		Finally, real convexity gives
		$|D_{\mathbb R}^2u|\le\Delta_{\mathbb R}u=4\operatorname{tr}A$.
		Hence, in the normalized coordinates,
		\[
		\|D_{\mathbb R}^2u\|_{L^\infty(\mathbb C^n)}
		\le4\|\operatorname{tr}A\|_{L^\infty(\mathbb C^n)}\le C_n.
		\]
		Undoing the fixed normalization in \eqref{e3_3} proves the proposition.
	\end{proof}
	
	\section{Liouville theorem}\label{sec:liouville}
	
	\begin{proof}[Proof of Theorem~\ref{thm:liouville}]
		Proposition~\ref{prop:globalbound} bounds $D_{\mathbb R}^2u$ globally.
		Since $D_{\mathbb R}^2u\ge0$, one has $A=(u_{i\bar j})\ge0$,
		and $\det A=1$ gives $A>0$.
		Hence there are constants $0<\lambda\le\Lambda<\infty$ such that
		\begin{equation}\label{e4_1}
			\lambda I\le A\le\Lambda I
			\quad\text{on }\mathbb C^n.
		\end{equation}
		
		We now place the equation in the standard uniformly elliptic setting.
		For a real symmetric $2n\times2n$ matrix written in block form as
		\[
		H=\begin{pmatrix}U&V\\V^T&W\end{pmatrix},
		\]
		define
		\[
		P(H)
		=\frac14\bigl(U+W+\sqrt{-1}(V-V^T)\bigr).
		\]
		Thus $P(D_{\mathbb R}^2u)=A$.
		The real--complex identity used in the proof of
		Lemma~\ref{lem:realcomplex} shows that $P(N)\ge0$ whenever
		$N\ge0$, and direct tracing gives
		\[
		\operatorname{tr}P(N)=\frac14\operatorname{tr}N.
		\]
		
		Define
		\[
		\widetilde F(H)
		=
		\inf_{\lambda I\le C_0\le\Lambda I}
		\left\{
		\log\det C_0+
		\operatorname{tr}
		\bigl(C_0^{-1}(P(H)-C_0)\bigr)
		\right\},
		\]
		where the infimum is taken over positive Hermitian matrices $C_0$.
		The function $\widetilde F$ is concave as an infimum of affine
		functions.
		By concavity of $\log\det$, for every admissible $C_0$,
		\[
		\log\det A
		\le
		\log\det C_0+
		\operatorname{tr}\bigl(C_0^{-1}(A-C_0)\bigr).
		\]
		Since $A$ itself is admissible by \eqref{e4_1}, equality is attained
		at $C_0=A$.
		Therefore
		\[
		\widetilde F(D_{\mathbb R}^2u)
		=\log\det A=0.
		\]
		For every admissible $C_0$ and every $N\ge0$,
		\[
		\frac1{4\Lambda}\operatorname{tr}N
		\le
		\operatorname{tr}(C_0^{-1}P(N))
		\le
		\frac1{4\lambda}\operatorname{tr}N.
		\]
		Adding $N$ to the argument of each affine function in the definition
		of $\widetilde F$ and then taking the infimum gives
		\[
		\frac1{4\Lambda}\operatorname{tr}N
		\le
		\widetilde F(H+N)-\widetilde F(H)
		\le
		\frac1{4\lambda}\operatorname{tr}N.
		\]
		Thus $\widetilde F$ is uniformly elliptic on the full space of
		real symmetric matrices, with ellipticity constants depending only
		on $\lambda$ and $\Lambda$.
		
		After subtracting an affine function, assume
		$u(0)=0$ and $D_{\mathbb R}u(0)=0$.
		Let $B_r$ denote the Euclidean ball of radius $r$ centered at the
		origin.
		For $R>1$, set
		\[
		u_R(z)=R^{-2}u(Rz).
		\]
		Then
		\[
		D_{\mathbb R}^2u_R(z)=D_{\mathbb R}^2u(Rz),
		\qquad
		\widetilde F(D_{\mathbb R}^2u_R)=0.
		\]
		The global Hessian bound and the normalization at the origin give
		a uniform $C^{1,1}$ bound for $u_R$ on $B_2$.
		The Evans--Krylov interior estimate for concave uniformly elliptic
		equations~\cite{r1} therefore yields constants
		$\alpha\in(0,1)$ and $C_1<\infty$, independent of $R$, such that
		\[
		|D_{\mathbb R}^2u_R(\zeta)-D_{\mathbb R}^2u_R(\eta)|
		\le C_1|\zeta-\eta|^\alpha,
		\qquad \zeta,\eta\in B_1.
		\]
		For fixed $x,y\in\mathbb C^n$ and all sufficiently large $R$,
		$x/R,y/R\in B_1$, and hence
		\[
		|D_{\mathbb R}^2u(x)-D_{\mathbb R}^2u(y)|
		\le C_1R^{-\alpha}|x-y|^\alpha.
		\]
		Letting $R\to\infty$ shows that $D_{\mathbb R}^2u$ is constant
		on $\mathbb C^n$.
		Therefore $u$ is a real quadratic polynomial.
	\end{proof}
	
	\begin{proof}[Proof of Corollary~\ref{cor:outside}]
		The Hermitian matrix $(u_{i\bar j})$ is nonsingular everywhere
		because its determinant is one.
		On $\mathbb C^n\setminus K_0$, real convexity gives
		$(u_{i\bar j})\ge0$, hence $(u_{i\bar j})>0$.
		The number of negative eigenvalues of a continuous nonsingular
		Hermitian matrix is locally constant.
		Since $\mathbb C^n$ is connected, it follows that
		$(u_{i\bar j})>0$ on all of $\mathbb C^n$.
		
		Choose $R>0$ so large that $K_0\subset B_R$.
		Then $D_{\mathbb R}^2u\ge0$ on $\partial B_R$, and
		Proposition~\ref{prop:convexity} gives $D_{\mathbb R}^2u\ge0$
		on $B_R$.
		Together with the hypothesis outside $B_R$, this gives
		$D_{\mathbb R}^2u\ge0$ on $\mathbb C^n$.
		Theorem~\ref{thm:liouville} now implies that $u$ is a real
		quadratic polynomial.
	\end{proof}
	\section*{Acknowledgement}
	The authors would like to thank Professors Xinan Ma, Guohuan Qiu, Xushan Tu and Jiaxiang Wang for helpful discussions.
	In particular, they thank Jiaxiang Wang for suggesting the use of the
	$C^2$ interior estimate in the proof of the Liouville theorem, while previously the same theorem was proved using a $C^3$
	interior estimate.

\begin{flushleft}
Hongyu Chen\\
School of Mathematics, Sichuan University, Chengdu, China\\
Email: \href{mailto:hongyu.chern@gmail.com}{\nolinkurl{hongyu.chern@gmail.com}}

\medskip
Jingchen Hu\\
School of Mathematics, Sichuan University, Chengdu, China\\
Email: \href{mailto:jingchenhu@scu.edu.cn}{\nolinkurl{jingchenhu@scu.edu.cn}}

\medskip
Li Sheng\\
School of Mathematics, Sichuan University, Chengdu, China\\
Email: \href{mailto:lsheng@scu.edu.cn}{\nolinkurl{lsheng@scu.edu.cn}}
\end{flushleft}


\begin{thebibliography}{EST+18}
		\setlength{\itemsep}{2pt}

		\bibitem[B03]{r2}
		Z. B\l ocki,
		Interior regularity of the degenerate Monge--Amp\`ere equation,
		\emph{Bull. Aust. Math. Soc.} \textbf{68} (2003), no.~1, 81--92.

		\bibitem[BD11]{r3}
		Z. B\l ocki and S. Dinew,
		A local regularity of the complex Monge--Amp\`ere equation,
		\emph{Math. Ann.} \textbf{351} (2011), 411--416.

		\bibitem[CC95]{r1}
		L.~A. Caffarelli and X. Cabr\'e,
		\emph{Fully Nonlinear Elliptic Equations},
		American Mathematical Society Colloquium Publications, vol.~43,
		American Mathematical Society, Providence, RI, 1995.

		\bibitem[CNS86]{r25}
		L. Caffarelli, L. Nirenberg, and J. Spruck,
		The Dirichlet problem for the degenerate Monge--Amp\`ere equation,
		\emph{Rev. Mat. Iberoamericana} \textbf{2} (1986),
		no.~1--2, 19--27.

		\bibitem[C58]{r4}
		E. Calabi,
		Improper affine hyperspheres of convex type and a generalization
		of a theorem by K. J\"orgens,
		\emph{Michigan Math. J.} \textbf{5} (1958), 105--126.

		\bibitem[CY10]{ChangYuan2010}
		S.-Y.~A. Chang and Y. Yuan,
		A Liouville problem for the Sigma-2 equation,
		\emph{Discrete Contin. Dyn. Syst.} \textbf{28} (2010),
		no.~2, 659--664.

		\bibitem[C26]{r5}
		B.-L. Chen,
		Gradient estimates and a Liouville theorem of complex
		Monge--Amp\`ere equations,
		\emph{Sci. China Math.} (2026),
		published online 10 February 2026.

		\bibitem[CLM26]{ChenLiMa2026}
		C. Chen, J. Li, and X.-N. Ma,
		Brunn--Minkowski inequality for the first complex
		$\sigma_2$-Hessian eigenvalue,
		arXiv:2606.25678 (2026).

		\bibitem[CHS26]{r6}
		H. Chen, J. Hu, and L. Sheng,
		Power convexity of solutions to the complex Monge--Amp\`ere
		equation $\det(u_{i\bar j})=1$ in complex dimension two,
		arXiv:2607.06895 (2026).

		\bibitem[CZZ26]{r35}
		R. Chen, X. Zhou, and R. Zhu,
		Interior $C^{2,\alpha}$ regularity for the quadratic Hessian
		equation,
		arXiv:2608.29484v2 (2026).

		\bibitem[CX23]{r9}
		J. Cheng and Y. Xu,
		Interior $W^{2,p}$ estimate for small perturbations to the
		complex Monge--Amp\`ere equation,
		\emph{Calc. Var. Partial Differential Equations}
		\textbf{62} (2023), no.~8, Paper No.~231.

		\bibitem[CY75]{r8}
		S.-Y. Cheng and S.-T. Yau,
		Differential equations on Riemannian manifolds and their
		geometric applications,
		\emph{Comm. Pure Appl. Math.} \textbf{28} (1975), 333--354.

		\bibitem[CW01]{ChouWang2001}
		K.-S. Chou and X.-J. Wang,
		A variational theory of the Hessian equation,
		\emph{Comm. Pure Appl. Math.} \textbf{54} (2001), no.~9,
		1029--1064.

		\bibitem[CR21]{ConlonRochon2021}
		R.~J. Conlon and F. Rochon,
		New examples of complete Calabi--Yau metrics on $\mathbb C^n$
		for $n\ge3$,
		\emph{Ann. Sci. \'Ec. Norm. Sup\'er.} (4)
		\textbf{54} (2021), no.~2, 259--303.

		\bibitem[EST+18]{AIM2018}
		G. Edwards, G. Sz\'ekelyhidi, V. Tosatti, et al.,
		AimPL: Nonlinear PDEs in real and complex geometry,
		American Institute of Mathematics workshop problem list, 2018,
		\url{http://aimpl.org/nonlinpdegeom}.

		\bibitem[H12]{r13}
		W. He,
		On the regularity of the complex Monge--Amp\`ere equations,
		\emph{Proc. Amer. Math. Soc.} \textbf{140} (2012), no.~5,
		1719--1727.

		\bibitem[H24a]{r15}
		J. Hu,
		A metric lower bound estimate for geodesics in the space of
		K\"ahler potentials,
		\emph{J. Geom. Anal.} \textbf{34} (2024), Paper No.~225.

		\bibitem[H24b]{r16}
		J. Hu,
		A maximum rank theorem for solutions to the homogenous complex
		Monge--Amp\`ere equation in a $\C$-convex ring,
		\emph{Calc. Var. Partial Differential Equations}
		\textbf{63} (2024), Paper No.~166.

		\bibitem[H25]{r14}
		J. Hu,
		The preservation of convexity by geodesics in the space of
		K\"ahler potentials on complex affine manifolds,
		\emph{Math. Ann.} \textbf{393} (2025), no.~2, 1635--1681.

		\bibitem[HS26]{r17}
		J. Hu and L. Sheng,
		Convexity of the potential function of the Einstein-K\"ahler
		metric on a convex domain,
		arXiv:2603.10530v1 (2026).

		\bibitem[J48]{r18}
		F. John,
		Extremum problems with inequalities as subsidiary conditions,
		in \emph{Studies and Essays Presented to R. Courant on his
		60th Birthday},
		Interscience, New York, 1948, 187--204.

		\bibitem[J54]{r19}
		K. J\"orgens,
		\"Uber die L\"osungen der Differentialgleichung $rt-s^2=1$,
		\emph{Math. Ann.} \textbf{127} (1954), 130--134.

		\bibitem[L91]{LeBrun1991}
		C. LeBrun,
		Complete Ricci-flat K\"ahler metrics on $\mathbb C^n$ need not be flat,
		in \emph{Several Complex Variables and Complex Geometry, Part 2},
		Proc. Sympos. Pure Math., vol.~52,
		American Mathematical Society, Providence, RI, 1991, 297--304.

		\bibitem[LS21]{r20}
		A.-M. Li and L. Sheng,
		A Liouville theorem on the PDE $\det(f_{i\bar j})=1$,
		\emph{Math. Z.} \textbf{297} (2021), 1623--1632.

		\bibitem[LRW16]{LiRenWang2016}
		M. Li, C. Ren, and Z. Wang,
		An interior estimate for convex solutions and a rigidity theorem,
		\emph{J. Funct. Anal.} \textbf{270} (2016), no.~7, 2691--2714.

		\bibitem[L19]{YangLi2019}
		Y. Li,
		A new complete Calabi--Yau metric on $\mathbb C^3$,
		\emph{Invent. Math.} \textbf{217} (2019), no.~1, 1--34.

		\bibitem[LW26]{r33}
		Z. Li and K. Wu,
		Interior Hessian estimates for the quadratic Hessian equation,
		arXiv:2608.23233v2 (2026).

		\bibitem[MSY19]{r21}
		M. McGonagle, C. Song, and Y. Yuan,
		Hessian estimates for convex solutions to quadratic Hessian
		equation,
		\emph{Ann. Inst. H. Poincar\'e Anal. Non Lin\'eaire}
		\textbf{36} (2019), no.~2, 451--454.

		\bibitem[P72]{r22}
		A.~V. Pogorelov,
		On the improper convex affine hyperspheres,
		\emph{Geom. Dedicata} \textbf{1} (1972), 33--46.

		\bibitem[P78]{r23}
		A.~V. Pogorelov,
		\emph{The Minkowski Multidimensional Problem},
		Halsted Press [John Wiley \& Sons], New York, 1978.

		\bibitem[QY26]{r36}
		G. Qiu and J. Yan,
		Interior curvature estimates for the graphical scalar curvature
		equation in all dimensions,
		arXiv:2609.02581v1 (2026).

		\bibitem[RS84]{r24}
		D. Riebesehl and F. Schulz,
		A priori estimates and a Liouville theorem for complex
		Monge--Amp\`ere equations,
		\emph{Math. Z.} \textbf{186} (1984), 57--66.

		\bibitem[S07]{Savin2007}
		O. Savin,
		Small perturbation solutions for elliptic equations,
		\emph{Comm. Partial Differential Equations}
		\textbf{32} (2007), no.~4--6, 557--578.

		\bibitem[SY20]{r27}
		R. Shankar and Y. Yuan,
		Hessian estimate for semiconvex solutions to the sigma-2 equation,
		\emph{Calc. Var. Partial Differential Equations}
		\textbf{59} (2020), no.~1, Paper No.~30.

		\bibitem[SY22]{ShankarYuan2022}
		R. Shankar and Y. Yuan,
		Rigidity for general semiconvex entire solutions to the
		$\sigma_2$ equation,
		\emph{Duke Math. J.} \textbf{171} (2022), no.~15, 3201--3214.

		\bibitem[SY25]{r28}
		R. Shankar and Y. Yuan,
		Hessian estimates for the sigma-2 equation in dimension four,
		\emph{Ann. of Math.} (2) \textbf{201} (2025), no.~2, 489--513.

		\bibitem[S19]{Szekelyhidi2019}
		G. Sz\'ekelyhidi,
		Degenerations of $\mathbb C^n$ and Calabi--Yau metrics,
		\emph{Duke Math. J.} \textbf{168} (2019), no.~14, 2651--2700.

		\bibitem[T06]{Tian2006}
		G. Tian,
		Aspects of metric geometry of four manifolds,
		in \emph{Inspired by S.~S. Chern},
		Nankai Tracts Math., vol.~11,
		World Scientific Publishing, Hackensack, NJ, 2006, 381--397.

		\bibitem[W13]{r29}
		Y. Wang,
		A Liouville theorem for the complex Monge--Amp\`ere equation,
		preprint, arXiv:1303.2403 (2013).

		\bibitem[WY09]{r30}
		M. Warren and Y. Yuan,
		Hessian estimates for the sigma-2 equation in dimension 3,
		\emph{Comm. Pure Appl. Math.} \textbf{62} (2009), no.~3,
		305--321.

		\bibitem[Y23]{r31}
		Y. Yuan,
		A monotonicity approach to Pogorelov's Hessian estimates for
		Monge--Amp\`ere equation,
		\emph{Math. Eng.} \textbf{5} (2023), no.~2, 1--6.

		\bibitem[ZZ26]{r34}
		X. Zhou and R. Zhu,
		Interior $C^{2,\alpha}$ regularity for convex solutions of
		the 2-Hessian equation,
		arXiv:2608.24604v1 (2026).
	
	\end{thebibliography}
\end{document}